\documentclass[11pt,reqno]{amsart}

\usepackage[utf8]{inputenc}

 \makeatletter
 \def\swappedhead#1#2#3{%
 \thmnumber{\@upn{\normalfont\@ifnotempty{#2}{(}#2\@ifnotempty{#2}{)~}}}%
  \thmname{#1}%
  \thmnote{ {\the\thm@notefont(#3)}}}

\def\@thm#1#2#3{%
  \ifhmode\unskip\unskip\par\fi
  \normalfont
  \trivlist
  \let\thmheadnl\relax
  \let\thm@swap\@gobble
  \let\thm@indent\noindent 
  \thm@headfont{\bfseries}
  \thm@notefont{\fontseries\mddefault\upshape}%
  \thm@headpunct{.}
  \thm@headsep 5\p@ plus\p@ minus\p@\relax
  \thm@space@setup
  #1
  \@topsep \thm@preskip               
  \@topsepadd \thm@postskip           
  \def\@tempa{#2}\ifx\@empty\@tempa
    \def\@tempa{\@oparg{\@begintheorem{#3}{}}[]}%
  \else
    \refstepcounter{#2}%
    \def\@tempa{\@oparg{\@begintheorem{#3}{\csname the#2\endcsname}}[]}%
  \fi
  \@tempa
}
 \makeatother

\usepackage{amssymb}
\RequirePackage{doi}
\usepackage{hyperref}
\usepackage[capitalize]{cleveref}
\usepackage{mathtools} 
\usepackage{babel}

\newtheorem{pfassump}{Assumption}

\swapnumbers  

\numberwithin{equation}{section}

\theoremstyle{plain}

\newtheorem{cor}[equation]{Corollary}
\newtheorem{lem}[equation]{Lemma}
\newtheorem{prop}[equation]{Proposition}
\newtheorem{thm}[equation]{Theorem}

\crefformat{lem}{Lemma~#2#1#3}
\Crefname{lem}{Lemma}{Lemmas}
\crefname{lem}{lemma}{lemmas}
\crefmultiformat{lem}{Lemmas~#2#1#3}{ and~#2#1#3}{, #2#1#3}{ and~#2#1#3}
\crefrangeformat{lem}{Lemmas~#3#1#4 to~#5#2#6}
\crefformat{cor}{Corollary~#2#1#3}
\Crefname{cor}{Corollary}{Corollaries}
\crefname{cor}{corollary}{corollaries}
\crefformat{prop}{Proposition~#2#1#3}
\Crefname{prop}{Proposition}{Propositions}
\crefname{prop}{proposition}{propositions}
\crefmultiformat{prop}{Propositions~#2#1#3}{ and~#2#1#3}{, #2#1#3}{ and~#2#1#3}
\crefformat{thm}{Theorem~#2#1#3}
\Crefname{thm}{Theorem}{Theorems}
\crefname{thm}{theorem}{theorems}
\crefmultiformat{thm}{Theorems~#2#1#3}{ and~#2#1#3}{, #2#1#3}{ and~#2#1#3}

\theoremstyle{definition}

\newtheorem{defn}[equation]{Definition}

\newtheorem{eg}[equation]{Example}
\newtheorem{notation}[equation]{Notation}
\newtheorem*{notation*}{Notation}
\newtheorem*{pfremark}{Remark}
\newtheorem{rem}[equation]{Remark}
\newtheorem{rems}[equation]{Remarks}

\crefmultiformat{rem}{Remarks~#2#1#3}{ and~#2#1#3}{, #2#1#3}{ and~#2#1#3}
\crefformat{rem}{Remark~#2#1#3}
\crefformat{rems}{Remark~#2#1#3}
\Crefname{rem}{Remark}{Remarks}
\crefname{rem}{remark}{remarks}

\crefformat{figure}{Figure~#2#1#3}

\newcounter{case}
\newcounter{caseholder} 
\numberwithin{case}{caseholder}
\newenvironment{case}[1][\unskip]{\refstepcounter{case}\bf\sffamily
\unskip\medskip \indent Case \thecase\ #1.\ \it}{\unskip\upshape}
\renewcommand{\thecase}{\arabic{case}}
\crefformat{case}{Case~#2#1#3}
\Crefname{case}{Case}{Cases}
\crefname{case}{case}{cases}

\newcounter{subcase}
\newenvironment{subcase}[1][\unskip]{\refstepcounter{subcase}\bf\sffamily
\medskip \indent \quad Subcase \thesubcase\ #1.\ \it}{\unskip\upshape}
\numberwithin{subcase}{case}
\crefformat{subcase}{Subcase~#2#1#3}
\crefname{subcase}{subcase}{subcases}
\Crefname{subcase}{Subcase}{Subcases}

\renewcommand{\pmod}[1]{\ (\mathrm{mod}~#1)}

\let\oldenumerate=\enumerate
\def\enumerate{\oldenumerate \itemsep=\smallskipamount}
\let\olditemize=\itemize
\def\itemize{\olditemize \itemsep=\smallskipamount}

\makeatletter 
\def\section{\goodbreak\vskip1cm\@startsection{section}{1}%
  \z@{.7\linespacing\@plus\linespacing}{1.5\linespacing}%
  {\normalfont\bfseries\larger[2]\centering}}
\def\subsection{\goodbreak\@startsection{subsection}{2}%
  \z@{1\linespacing\@plus0.25\linespacing}{0.5\linespacing}%
  {\normalfont\bfseries\centering}}
\renewenvironment{abstract}{%
  \ifx\maketitle\relax
    \ClassWarning{\@classname}{Abstract should precede
      \protect\maketitle\space in AMS document classes; reported}%
  \fi
  \global\setbox\abstractbox=\vtop \bgroup
    \normalfont 
    \vskip5mm 
    \list{}{\labelwidth\z@
      \leftmargin3pc \rightmargin\leftmargin
      \listparindent\normalparindent \itemindent\z@
      \parsep\z@ \@plus\p@
      
    }%
    \item[\hskip\labelsep\scshape\abstractname.]%
}{%
  \endlist\egroup
  \ifx\@setabstract\relax \@setabstracta \fi
}
\makeatother

\newcommand{\cartprod}{\mathbin{\raise0.7pt\hbox{\smaller[2]$\square$}}} 

\newcommand{\bx}[1]{\widetilde{#1}} 
\newcommand{\invol}{\mathsf{z}} 
\newcommand{\twin}{\invol}
\newcommand{\ZZ}{\mathbb{Z}}

\newcommand{\iso}{\cong}

\DeclareMathOperator{\Aut}{Aut}
\DeclareMathOperator{\Cay}{Cay}

\DeclareMathOperator{\Qd}{Qd}
\DeclareMathOperator{\Tr}{Tr}

\newcommand{\ccf}[1]{(cf.~\cref{#1})}
\newcommand{\csee}[1]{(see \cref{#1})}
\newcommand{\fullcsee}[2]{(see \fullcref{#1}{#2})}
\newcommand{\pref}[1]{\textup(\ref{#1}\textup)}
\newcommand{\fullref}[2]{\ref{#1}\pref{#1-#2}}
\newcommand{\fullcref}[2]{\cref{#1}\pref{#1-#2}}

\makeatletter
\newcommand{\noprelistbreak}{\smallskip\@nobreaktrue\nopagebreak} 
\makeatother

\begin{document}

\title[Automorphisms of the canonical double cover of a toroidal grid]%
{\larger[2]Automorphisms of the canonical  
\\[2pt] double cover of a toroidal grid} 

\author{Dave Witte Morris}
\address{Department of Mathematics and Computer Science, University of Lethbridge, Lethbridge, Alberta, T1K~3M4, Canada}
\email{dmorris@deductivepress.ca,
https://deductivepress.ca/dmorris}

\date{\today}

\begin{abstract}
The Cartesian product of two cycles ($C_n \cartprod C_m$) has a natural embedding on the torus, such that each face of the embedding is a $4$-cycle. 
The toroidal grid $\Qd(m,n,r)$ is a generalization of this in which there is a shift by~$r$ when traversing the meridian of length~$m$. 

In 2008, Steve Wilson found two interesting infinite families of (nonbipartite) toroidal grids that are unstable. (By definition, this means that the canonical bipartite double cover of the grid has more than twice as many automorphisms as the grid has.) It is easy to see that bipartite grids are also unstable, because the canonical double cover is disconnected. Furthermore, there are degenerate cases in which there exist two different vertices that have the same neighbours. This paper proves Wilson's conjecture that $\Qd(m,n,r)$ is stable for all other values of the parameters. 

In addition, we prove an analogous conjecture of Wilson for the triangular grids $\Tr(m,n,r)$ that are obtained by adding a diagonal to each face of $\Qd(m,n,r)$ (with all of the added diagonals parallel to each other).
\end{abstract}

\maketitle

\vskip-5mm 

\section{Introduction}

\let\myoldsection=\section
\renewcommand{\section}{\newpage\thispagestyle{plain}\myoldsection}

\begin{defn}
The \emph{canonical bipartite double cover} \cite{CanCover} of a graph~$X$ is the bipartite graph~$BX$ with $V(BX) = V(X) \times \{0,1\}$, where
	\[ \text{$(v,0)$ is adjacent to $(w,1)$ in $BX$}
	\quad \iff \quad
	\text{$v$ is adjacent to~$w$ in~$X$} . \]
Letting $S_2$ be the symmetric group on the $2$-element set~$\{0,1\}$, it is clear that $\Aut X \times S_2$ is a subgroup of $\Aut BX$.
 If this subgroup happens to be all of $\Aut BX$, then we say that $X$ is \emph{stable} \cite[p.~160]{MarusicScapellatoSalvi}.
\end{defn}

Understanding unstable graphs is a fundamental problem in the study of automorphisms of direct products (see \cite[Prop.~5.6]{Morris-AutDirProd}), and also arises in other contexts (see the introductions of \cite{QinXiaZhou-circulant} and~\cite{Wilson}).

In the appendix of~\cite{Wilson}, S.\,Wilson stated conjectures about exactly which graphs in certain families are unstable. Here is the current status of each of these conjectures:
\smallskip 
	\begin{enumerate} \renewcommand{\theenumi}{A.\arabic{enumi}} \itemsep=\smallskipamount
	\item Wilson's conjecture about circulant graphs is known to be false: a counterexample was published in~\cite[p.~156]{QinXiaZhou-circulant}, and infinite families of additional counterexamples can be found in~\cite{HMMaut}. We still do not know which circulant graphs are unstable, but progress was made in \cite{FernandezHujdurovic,HMMaut,HMMlow,QinXiaZhou-circulant}.
	\item Wilson's conjecture on generalized Petersen graphs is correct~\cite{QinXiaZhou-Petersen}.
	\item There does not seem to have been any progress on Wilson's conjecture about rose window graphs.
	\item This paper proves (slight generalizations of) Wilson's two conjectures about toroidal graphs.
	\end{enumerate}

\bigbreak 

We now state our main results.

\begin{defn}[{\cite[pp.~380 and 381]{Wilson}}] \label{QdDefn}
Given $m,n \in \ZZ$ (with $m,n \ge 2$), and $r \in \ZZ_n$, we can 
	\begin{itemize}
	\item number the vertices of the cycle~$C_n$ with the elements of~$\ZZ_n$, 
	and 
	\item number the vertices of the path~$P_{m + 1}$ with the elements of $\{0,1,\ldots, m\}$. 
	\end{itemize}
(In the special case where $m = 2$, we let $C_2 = K_2$.) Then 
	\begin{enumerate}
	\item $\Qd(m,n,r)$ is the graph that is obtained from the Cartesian product $C_n \cartprod P_{m + 1}$ by identifying the vertex $(x,m)$ with $(x + r, 0)$, for each $x \in \ZZ_n$,
	and
	\item $\Tr(m,n,r)$ is the graph that is obtained from $\Qd(m,n,r)$ by adding an edge from $(x, y)$ to $(x+1, y - 1)$ for each $x \in \ZZ_n$ and $y \in \{1,2,\ldots, m\}$.
	\end{enumerate}
(See \cref{GridIsCayley} for reformulations of these definitions in the language of Cayley graphs.)
\end{defn}

\begin{rem}
$\Qd(m,n,r)$ has a natural embedding on the torus, such that each face of the embedding is a $4$-cycle.
 (In the special case where $r = 0$, the graph $\Qd(m,n,0)$ is isomorphic to the Cartesian product $C_n \cartprod C_m$.) So $\Qd(m,n,r)$ is often called a \emph{toroidal grid}. The graph $\Tr(m,n,r)$ is obtained by adding a diagonal in each face of $\Qd(m,n,r)$. Therefore, the faces of its natural toroidal embedding are triangles.
 \end{rem}
 
 There are some trivial reasons for a graph to be unstable \cite[p.~360]{Wilson}:
 	\begin{enumerate}
	\item Every disconnected graph is unstable.
	\item Every bipartite graph with a nontrivial automorphism is unstable.
	\item If two different vertices of a graph have the same neighbours, then the graph is unstable. (These are called ``twin vertices'' \cite{KotlovLovasz-twinfree}.)
	\end{enumerate}
This motivates the following definition:
 
 \begin{defn}[{\cite[p.~360]{Wilson}}]
 An unstable graph is \emph{nontrivially unstable} if it is connected and nonbipartite, and has no twin vertices. (Otherwise, it is \emph{trivially unstable}.)
 \end{defn}

The following two results were conjectured by S.\,Wilson \cite[pp.~380 and~381]{Wilson}, who proved the direction~($\Leftarrow$) of each \lcnamecref{Qd} (except parts \pref{Tr-2k/4/pm1} and~\pref{Tr-4/23/pm1} of~\ref{Tr}). 
(We write ``$\pm$'' in \cref{Qd} because $\Qd(m,n,r)$ is always isomorphic to $\Qd(m,n,-r)$, as explained in \cref{abyc}. Also recall that the parameter~$r$ in $\Qd(m,n,r)$ and $\Tr(m,n,r)$ is taken modulo~$n$.)

\begin{thm}[cf.\ \cref{val4}] \label{Qd}
$\Qd(m,n,r)$ is nontrivially unstable if and only if it is:
	\begin{enumerate} 
	\item \label{Qd-m/4k/k}
	$\Qd(m, 4k, \pm k)$, where $m + k$ is odd,
	or
	\item \label{Qd-kl}
	$\Qd(2m, km, \pm 4\ell m)$ \textup($\iso \Qd(m, 2km, \pm 2\ell m)$ if $m > 1$\textup), where $m$ is odd, $4\ell^2 \equiv \pm1 \pmod{k}$, and either $m > 1$ or $2\ell \not\equiv \pm 1 \pmod{k}$.
	\end{enumerate}
\end{thm}

\begin{thm}[cf.\ \cref{val6}] \label{Tr}
$\Tr(m,n,r)$ is nontrivially unstable if and only if it is:
	\begin{enumerate}
	\item \label{Tr-4/2k/2}
	$\Tr(2, 4k, 4) \iso \Tr(2, 4k, -2) \iso \Tr(4, 2k, 2)$,
	or
	\item \label{Tr-2/4k/2k+1}
	$\Tr(2, 4k, 2k + 1)$, 
	or
	\item \label{Tr-2k/4/2}
	$\Tr(2, 4k, 2k) \iso \Tr(2, 4k, 2k + 2) \iso \Tr(2k, 4, 2)$,
	or
	\item \label{Tr-4/4k/0}
	$\Tr(4, 2k, 0) \iso \Tr(4, 2k, 4) \iso \Tr(2k, 4, 0)$, 
	or
	\item \label{Tr-2k/4/pm1}
	$\Tr(2k, 4, 1)$, with $k > 1$, 
	or\/
	$\Tr(2k, 4, -1)$, or\/
	\item \label{Tr-4/23/pm1}
	$\Tr(4, 2, 1)$, or\/
	$\Tr(4, 3, -1)$.
	\end{enumerate}
\textup(If $k$ is odd, then the graphs in \pref{Tr-2k/4/2} are isomorphic to the graphs in~\pref{Tr-4/4k/0}. If $k$ is even, then the two graphs in \pref{Tr-2k/4/pm1} are isomorphic to each other.\textup)
\end{thm}

\begin{eg}
By searching \cref{Qd} for cases where $r = 0$, we see that, for $n \ge m \ge 2$, the Cartesian product $C_n \cartprod C_m$ is nontrivially unstable if and only if $n = 2m$ and $m$~is odd. 
\end{eg}

\begin{rem} \label{GridTriv}
\Cref{Qd,Tr} only list graphs that are \emph{nontrivially} unstable.
However, it is easy to check whether a toroidal grid is trivially unstable. First, note that they are all connected. Also:
	\begin{enumerate} \renewcommand{\theenumi}{\alph{enumi}}
	\item \label{GridTriv-QdBip}
	$\Qd(m, n, r)$ is bipartite if and only if $n$ and~$m + r$ are even.
	\item $\Qd(m, n, r)$ has twin vertices if and only if $m = 2$ and $r = \pm 2$ \ccf{val4triv}).
	\item $\Tr(m, n, r)$ is never bipartite (because it has triangles).
	\item \label{GridTriv-TrTwin}
	$\Tr(m, n, r)$ has twin vertices if and only if $(m,n,r) \in \{(2,4,1), (3,3,0)\}$ \ccf{val6triv}.
	\end{enumerate}
\end{rem}

\begin{rems} \leavevmode
\noprelistbreak
	\begin{enumerate}
	
	\item The assumption that $m,n \ge 2$ is not stated explicitly in~\cite{Wilson}. Wilson's conjectures also seem to implicitly assume that $\gcd(n, r) \neq 1$. We do not make this assumption, so \cref{Qd,Tr} include include infinite families of graphs that are not listed in Wilson's conjectures.
	

	\item There are other differences between Wilson's conjectures \cite[pp.~380 and~381]{Wilson} and our statements of the results. In particular:
		\begin{enumerate}
		\item Wilson omits $\Qd(m, 4k, -k)$ and $\Qd(2m, km, \pm 4\ell)$, and usually omits $\Tr(m,n, m - r)$ when $\Tr(m,n,r)$ is listed. \Cref{abyc} explains that they are alternate representations of other graphs in the list.
	
	\item Wilson requires $k$ to be odd in \fullref{Qd}{kl}, but we omit this redundant condition: it is a consequence of the equation $4\ell^2 \equiv \pm1 \pmod{k}$.
		\item Wilson uses $4k$ in \fullref{Tr}{4/4k/0}, instead of~$2k$. That eliminates the overlap with~\fullref{Tr}{2k/4/2}.
		\end{enumerate}
	
	\item The two occurrences of~``$\pm$'' in \fullref{Qd}{kl} cause redundancy (and could therefore be omitted), because $-\ell$ satisfies the congruence $4\ell^2 \equiv \pm1 \pmod{k}$ whenever $\ell$~does.
	\end{enumerate}
\end{rems}

\begin{rem}
Following a section of preliminaries, \cref{val4} is proved in \cref{val4sect}, and \cref{val6} is provided in \cref{val6sect}.
\Cref{val4,val6} are slightly more general than \cref{Qd,Tr}. For example, their statements in the language of abelian Cayley graphs allow for the case where $m$ is equal to~$1$. See the well-known \cref{GridIsCayley} for the translation between the two languages.
\end{rem}

\section{Preliminaries}

All graphs in this paper are simple (no loops or multiple edges).

\subsection{Abelian Cayley graphs}

\begin{defn}[cf.\ {\cite[p.~34]{GodsilRoyle}}] \label{CayleyDefn}
Let $S$ be a subset of an additive abelian group~$G$, such that $S = -S$ and $0 \notin S$. The corresponding \emph{abelian Cayley graph} $\Cay(G; S)$ is the graph whose vertices are the elements of~$G$, and with an edge joining the vertices $g$ and~$h$ if and only if $g = h + s$ for some $s \in S$.
\end{defn}

\begin{rem}
The adjective ``abelian'' in ``abelian Cayley graph'' is to emphasize the assumption that $G$ is abelian, so we will sometimes omit it when it is not relevant.
(The usual definition of $\Cay(G; S)$ does not require $G$ to be abelian, but we have no need for the nonabelian case in this paper.)
\end{rem}

The following simple (and well known) observation notes that the toroidal grids $\Qd(m,n,r)$ and $\Tr(m,n,r)$ are isomorphic to abelian Cayley graphs. The minus sign in $\Tr(m,n,-r)$ is because the definition of $\Tr(m,n,-r)$ would naturally identify it with the Cayley graph having $a - b$ as the third generator, but, for our purposes, it is more convenient to use $a + b$.

\begin{lem} \label{GridIsCayley}
Given $m,n,r \in \ZZ$ \textup(with $m,n \ge 2$\textup), let 
	\[ G =  \langle\, a, b \mid ma = rb, \ nb = 0, \ a + b = b + a \,\rangle , \]
so $G$ is an abelian group of order~$mn$. Then
	\[ \Qd(m,n, r) 
		\iso \Cay(G; \pm a, \pm b) \]
and 
	\[ \Tr(m,n,-r) 
		\iso \Cay \bigl( G; \pm a, \pm b, \pm(a + b) \bigr)
	. \]
\end{lem}

\begin{cor} \label{abyc}
$\Qd(m,n, r) \iso \Qd(m,n, -r)$ and\/ $\Tr(m,n,r) \iso \Tr(m,n, m - r)$.
\end{cor}

\begin{proof}
($\Qd$) We have $ma = -r(-b)$ and $n(-b) = 0$, so using $-b$ in the place of~$b$ yields a representation of $\Cay(G; \pm a, \pm b)$ as $\Qd(m,n, -r)$.

($\Tr$) Let $c = -(a + b)$, so $\{\pm a, \pm b, \pm(a + b)\} = \{\pm c, \pm b, \pm (c + b) \}$. Then
	\[ mc = -ma - mb = -rb - mb = -(r + m) b , \]
so using $c$ in the place of~$a$ yields a representation of $\Cay(G; \pm a, \pm b, \pm(a + b))$ as $\Tr(m,n, r + m)$. Therefore $\Tr(m,n, -r) \iso \Tr(m,n, r + m)$.
\end{proof}

\begin{rem}
By replacing $a$ with~$-a$, the proof of \cref{abyc} shows that $\Qd(m,n,r) \iso \Qd(m,n,-r)$. However, the same trick does not work for $\Tr(m,n,r)$: if $a$ is replaced with $-a$, then the equation $a + b + c = 0$ forces $b$ and~$c$ to also be replaced with their negatives. Since $m(-a) = -ma = -rb = r(-b)$, this does not result in any change in the parameter~$r$.
\end{rem}

Here is an abelian Cayley graph that appears in the statement of \fullcref{val4}{MoebiusPrism}:

\begin{notation}
$M_{2n} = \Cay(\ZZ_{2n}; \pm 1, n)$ is the Moebius ladder with $2n$ vertices.
\end{notation}

\subsection{Some classes of stable/unstable abelian Cayley graphs}

\begin{thm}[Morris {\cite[Thm.~1.1]{Morris-AutDirProd}}] \label{odd}
There are no nontrivially unstable abelian Cayley graphs of odd order.
\end{thm}

Recall that if $G$ is cyclic, then $\Cay(G; S)$ is a \emph{circulant} graph.
The following result is stated only for circulant graphs in~\cite{HMMlow}, but exactly the same proof applies to abelian Cayley graphs. 

\begin{prop}[Hujdurović-Mitrović-Morris,  {cf.\ \cite[Prop.~4.2]{HMMlow}}]
\label{val3}
There are no nontrivially unstable abelian Cayley graphs of valency $\le 3$.
\end{prop}

\begin{thm}[Hujdurović-Mitrović-Morris {\cite[Thm.~4.3]{HMMlow}}] \label{val4circulant}
A circulant graph\/ $\Cay(\ZZ_n, \{\pm a, \pm b\})$ of valency~$4$ is unstable if and only if either it is trivially unstable, or one of the following conditions is satisfied\/ \textup(perhaps after interchanging $a$ and~$b$\textup)\textup:
\noprelistbreak
	\begin{enumerate}
	\item \label{val4circulant-gcd}
	$n$ is divisible by~$8$ and $\gcd \bigl( |a|, |b| \bigr) = 4$,
	or
	\item \label{val4circulant-k2=1}
	$n \equiv 2 \pmod{4}$, $\gcd(b,n) = 1$, and $a \equiv \ell b + (n/2) \pmod{n}$, for some $\ell \in \ZZ$, such that $\ell^2 \equiv \pm1 \pmod{n}$.
	\end{enumerate}
\end{thm}

\begin{thm}[Hujdurović-Mitrović-Morris {\cite[Thm.~5.1]{HMMlow}}] \label{val5circulant}
A circulant graph\/ $\Cay(\ZZ_n; S)$ of valency~$5$ is unstable if and only if either it is trivially unstable, or it is either:
\noprelistbreak
	\begin{enumerate}
	\item \label{val5circulant-12k}
	$\Cay(\ZZ_{12k}; \pm s, \pm 2k, 6k )$ with $s$~odd, 
	or
	\item \label{val5circulant-8}
	$\Cay(\ZZ_8; \pm 1, \pm 3, 4 )$
	\end{enumerate}
\end{thm}

\begin{thm}[Hujdurović-Mitrović-Morris {\cite[Cor.~6.8]{HMMlow}}] \label{val6circulant}
A circulant graph 
	\[ X = \Cay(\ZZ_n, \{\pm a, \pm b, \pm c \}) \]
of valency\/~$6$ is unstable if and only if either it is trivially unstable, or it is one of the following:
\noprelistbreak
	\begin{enumerate}
	
	\setcounter{enumi}{0}
	
		\item \label{val6circulant-C1Se2}
		$\Cay(\ZZ_{8k},\{\pm a, \pm b, \pm 2k\})$, where $a$ and~$b$ are odd, 
	
		\item \label{val6circulant-C1Se4}
		$\Cay(\ZZ_{4k},\{\pm a, \pm b, \pm b + 2k\})$, where $a$ is odd and $b$ is even, 
		
		\item \label{val6circulant-C2}
		$\Cay \bigl( \ZZ_{4k}, \bigl\{ \pm a, \pm (a + k), \pm(a - k) \bigr\} \bigr)$, where $a \equiv 0 \pmod{4}$ and $k$ is odd, 
		
 		\item \label{val6circulant-C3order2}
                $\Cay(\ZZ_{8k},\{\pm a, \pm b, 4k \pm b\})$, where $a$ is even and $|a|$ is divisible by~$4$, 

        \item \label{val6circulant-C3order4}
                $\Cay(\ZZ_{8k},\{\pm a, \pm k, \pm 3k\})$, where $a \equiv 0 \pmod{4}$ and $k$ is odd, 

        		\item \label{val6circulant-C4}
		$\Cay(\ZZ_{4k},\{\pm a, \pm b, \pm m b + 2k \})$, where 
		\[ \text{$\gcd(m, 4k) = 1$, \quad $(m-1)a \equiv 2k \pmod{4k}$, \quad and} \]
		\[ \text{either \ 
			$m^2 \equiv 1 \pmod{4k}$
			\ or \ 
			$(m^2 + 1) b \equiv 0 \pmod{4k}$,} \]
		
		 \item \label{val6circulant-C4more}
           $\Cay(\ZZ_{8k},\{\pm a, \pm b, \pm c \})$, where there exists $m \in \ZZ$, such that
           \[ \text{$\gcd(m, 8k) = 1$, \quad $m^2 \equiv 1 \pmod{8k}$, \quad and} \]
        \[ (m-1)a \equiv (m + 1)b \equiv (m + 1)c \equiv 4k \pmod{8k} . \]
		  
	\end{enumerate}
\end{thm}

\subsection{Criteria for stability or instability}

\begin{lem}[cf.\ {\cite[Lem.~2.4]{FernandezHujdurovic}}] \label{StableIffStabilizer}
A connected, abelian Cayley graph $X = \Cay(G; S)$ is unstable if and only if there exists $\alpha \in \Aut BX$, such that $\alpha(0,0) = (0,0)$, but $\alpha(0,1) \neq (0,1)$.
\end{lem}

\begin{rem} \label{twiniff}
It is easy to see (and well known) that an abelian Cayley graph $\Cay(G; S)$ has twin vertices if and only if $S + \twin = S$, for some nonzero $\twin \in G$. In other words, $S$ is a union of cosets of~$\langle \twin \rangle$. By passing to a subgroup of~$\langle \twin \rangle$, there is no harm in assuming that $|\twin|$ is prime.
\end{rem}

The following result was stated only for circulant graphs in~\cite[Prop.~3.7]{HMMaut} (which is a slight generalization of \cite[Thm.~C.4]{Wilson}), but the same proof applies more generally. (In fact, the proof even applies without the assumption that $\invol$ has order~$2$, if $S + \invol = -(S + \invol)$ is symmetric. And there is no need for $G$ to be abelian.)

\begin{lem}[{\cite[Prop.~3.7]{HMMaut}, cf.\ \cite[Thm.~C.4]{Wilson}}] \label{IsoS+z}
An abelian Cayley graph $\Cay(G;S)$ is unstable if\/ $\Cay(G; S) \iso \Cay(G; S + \invol)$, for some element~$\invol$ of order~$2$ in~$G$.
\end{lem}

\begin{lem}[Wilson {\cite[\S2.2]{Wilson}}] \label{LargelyFixing}
A graph~$X$ is unstable if it has an automorphism~$\alpha$, such that the subgraph induced by the set of un-fixed vertices is disconnected and has a component~$C$, such that $C$ is bipartite, and either $\alpha(C) \neq C$ or each of the two bipartition sets of~$C$ is $\alpha$-invariant.
\end{lem}

\begin{prop}[Hujdurović-Mitrović {\cite[Prop.~3.2]{HujdurovicMitrovic}}] \label{triangles}
Let $X$ be a connected graph with more than one vertex, and assume that $X$ satisfies the following conditions:
	\begin{enumerate}
	\item Every edge of $X$ lies on a triangle.
	\item For every $x \in V(X)$, it holds that:
		\begin{enumerate}
		\item every vertex at distance~$2$ from~$x$ has a neighbour at distance~$3$ from~$x$, 
		and
		\item every vertex at distance~$3$ from~$x$ has a neighbour at distance~$4$ from~$x$.
		\end{enumerate}
	\end{enumerate}
Then $X$ is stable.
\end{prop}

\subsection{Other results on automorphisms and isomorphisms}

\begin{prop}[Baik-Feng-Sim-Xu {\cite[Thm.~1.1]{BaikFengSimXu4}}] \label{4cyclenormal}
Let $S$ be a generating set of an abelian group~$G$, such that $S = -S$, $0 \notin S$, and, for all $s,t,u,v \in S$:
	\[ s + t = u + v \neq 0 \implies \{s,t\} = \{u,v\} .\]
If $\alpha$ is any automorphism of the graph $\Cay(G; S)$, such that $\alpha(0) = 0$, then $\alpha$ is an automorphism of the group~$G$ \textup(i.e., $\alpha(g + h) = \alpha(g) + \alpha(h)$, for all $g,h \in G$\textup).
\end{prop}

\begin{defn}[{\cite[p.~35]{ProductHandbook}}]
Recall that the \emph{Cartesian product} $X \cartprod Y$ of two graphs $X$ and~$Y$ has vertex set $V(X) \times V(Y)$, and two vertices $(x_1, y_1)$ and $(x_2, y_2)$ are adjacent if and only if either
	\begin{itemize}
	\item $x_1 = x_2$ and $y_1y_2 \in E(Y)$,
	or
	\item $y_1 = y_2$ and $x_1x_2 \in E(X)$.
	\end{itemize}
\end{defn}

\begin{prop}[cf.\ {\cite[Thm.~6.10, p.~69]{ProductHandbook}}] \label{Aut(C4timesX)}
Let $X$ be a connected graph. If there does not exist a graph~$Y$\!, such that $X \iso Y \cartprod K_2$, then 
	\[ \text{$\Aut(X \cartprod K_2) = \Aut X \times S_2$
	\ and \ 
	$\Aut(X \cartprod C_4) = \Aut X \times \Aut C_4$.} \]
\end{prop}

We will use the following elementary observation in part~\pref{val4-MoebiusPrism} of the proof of \cref{val4unstable}.

\begin{lem} \label{B(XboxBip)}
Let $X$ and~$Y$ be graphs. If $Y$ is bipartite, then $B(X \cartprod Y) \iso (BX) \cartprod Y$.
\end{lem}

\subsection{Stability of a few specific graphs}

\begin{eg}[{\cite[Example~2.2]{QinXiaZhou-circulant}}] \label{KnStable}
If $n \ge 3$, then the complete graph $K_n$ is stable. (But $K_2$ is bipartite, and is therefore unstable.)
\end{eg}

\begin{lem} \label{SmallZnxZ2}
For $2 \le n \le 7$, the abelian Cayley graph
	\[ \Cay \bigl( \ZZ_n \times \ZZ_2; \pm(1,0), \pm(1,1), (0,1) \bigr) \]
is stable, unless $n = 4$, in which case it is unstable.
\end{lem}

\begin{proof}
This can be checked very quickly by computer. For example, the \textsf{sagemath} program in \cref{ZnxZkSage} can be executed on \url{https://cocalc.com}. (The program also verifies \cref{Z3xZn}.)
\end{proof}

\begin{rem}
Most cases of \cref{SmallZnxZ2} can be settled quite easily without a computer:
	\begin{itemize}
	\item If $n$ is odd, then $\ZZ_n \times \ZZ_2$ is cyclic, so \cref{val5circulant} can be applied. 
	\item If $n = 2$, then the Cayley graph is~$K_4$, which is stable by \cref{KnStable}.
	\item If $n = 4$, then part~\pref{val6-4} of the proof of \cref{val6unstable} explains why the Cayley graph is unstable.
	\end{itemize}
Therefore, $n = 6$ is the only case that requires effort (or a computer).
\end{rem}

\begin{lem} \label{Z3xZn}
For $3 \le n \le 12$, the abelian Cayley graph
	\[ \Cay \bigl( \ZZ_n \times \ZZ_3; \pm(1,0), \pm(1,1), \pm(0,1) \bigr) \]
is stable, unless $n = 3$, in which case it is unstable \textup(and is listed in \fullcref{val6}{9}\textup).
\end{lem}

\begin{proof}
As mentioned in the proof of \cref{SmallZnxZ2}, the stability/instability of these graphs is calculated by the \textsf{sagemath} program in \cref{ZnxZkSage}.

For $n = 3$, the elements $a = (1,0)$ and $b = (0,1)$ have order~$3$. 
Also, if we let $c =  - (1,1)$, then $a + b + c = (0,0)$.
Therefore, the Cayley graph is described in \fullcref{val6}{9}.
\end{proof}

\vskip0pt plus 0.3fil 
\begin{figure}[ht]
\begin{minipage}{11.5cm} 
\begin{verbatim}
for n in range(2, 13):
    for k in [2, 3]:
        G = direct_product_permgroups(
            [CyclicPermutationGroup(n), 
               CyclicPermutationGroup(k)])
        a, b = G.gens()
        assert {a.order(), b.order()} == {n, k}
        X = Graph(G.cayley_graph(generators=[a, b, a*b]))
        AutX = X.automorphism_group()
        K2 = graphs.CompleteGraph(2)
        BX = X.categorical_product(K2)
        AutBX = BX.automorphism_group()
        if 2 * AutX.order() != AutBX.order():
            print(n, k, "unstable")
\end{verbatim}
\end{minipage}
\caption{A \textsf{sagemath} \cite{sage} program to verify \cref{SmallZnxZ2,Z3xZn}.}
\label{ZnxZkSage}
\end{figure}

\section{Unstable abelian Cayley graphs of valency 4} \label{val4sect}

This \lcnamecref{val4sect} proves the following \lcnamecref{val4}, which implies \cref{Qd}. 
It also generalizes \cref{val4circulant}, which handles the case where $G$ is cyclic;
however, our argument relies on \cref{val4circulant}, so we are not providing an independent proof of that result.

\begin{thm} \label{val4}
A connected abelian Cayley graph\/ $\Cay(G; S)$ of valency~$4$ is unstable if and only if either it is bipartite, or it is in the following list \textup(up to a group isomorphism\,\textup)\textup:
\noprelistbreak
	\begin{enumerate}
	
	\item \label{val4-int4}
	$\Cay(G; \pm a, \pm b)$, where $|\langle a \rangle \cap \langle b \rangle| = 4$.

	\item \label{val4-square}
	$\Cay(G; \pm a, \pm b)$, where $|G : \langle b \rangle| = m$, $ma = 2 \ell m b$, $|b| = 2 k m$,
		and $4\ell^2 \equiv \pm 1 \pmod{k}$.
	
	\item \label{val4-2a=2b}
	$\Cay(G; \pm a, \pm b)$, where $2a = 2b$.
	
	\item \label{val4-MoebiusPrism} $\Cay \bigl( \ZZ_{2n} \times \ZZ_2; \pm(1,0), (n,0), (0,1) \bigr) \iso M_{2n} \cartprod K_2$.
		\begin{enumerate}
		\item \label{val4-MoebiusPrism-n=4}
		If $n = 2$, this is isomorphic to $K_4 \cartprod K_2$, and can also be realized as
			\[ \Cay \bigl( \ZZ_2 \times \ZZ_2 \times \ZZ_2; (1,0,0), (0,1,0), (1,1,0), (0,0,1) \bigr) . \]
		\end{enumerate}

	\end{enumerate}
\end{thm}

Before proving this \lcnamecref{val4}, let us show that it implies \cref{Qd}.

\begin{proof}[\bf Proof of \cref{Qd}]
The graphs in \cref{Qd} are required to be nontrivially unstable. Therefore, we see from \cref{val3} that they must have valency~$4$. Hence, it suffices to show that the graphs in \cref{Qd} are precisely the graphs that arise from \cref{val4} by applying \cref{GridIsCayley} (and are not trivially unstable), and satisfy the additional assumption that $m,n \ge 2$ (where $n = |b|$ and $m = |G : \langle b \rangle|$).
To do this, we consider each part of the statement of \cref{val4} individually. We also find the toroidal grids that are obtained by applying \cref{GridIsCayley} after interchanging $a$ and~$b$. (And we know from \cref{abyc} that $\Qd(m,n,r) \iso \Qd(m,n,-r)$.)
	\begin{itemize}
	
	\item[\pref{val4-int4}] The conditions in \fullref{Qd}{m/4k/k} that $n = 4k$ and $r = \pm k$ are a direct translation of the fact that $|\langle a \rangle \cap \langle b \rangle| = 4$. The additional condition that $m + k$ is odd ensures that the grid is not trivially unstable \csee{GridTriv}.
	
	Since the condition in \fullref{val4}{int4} is symmetric in $a$ and~$b$, no additional examples are obtained by interchanging $a$ and~$b$.
	
	\item[\pref{val4-square}] The grid $\Qd(m, 2km, 2 \ell m)$ of~\fullref{Qd}{kl} is obtained from a direct translation of the conditions in~\fullref{val4}{square}. The condition that $m$ is odd ensures that the grid is not trivially unstable \csee{GridTriv}. However, the definition of~$\Qd(m,n,r)$ requires $m > 1$.
		
	Now, we let $a$ play the role of~$b$ in \cref{GridIsCayley}. 	
	Note that $\gcd(k, \ell) = 1$, because $4\ell^2 \equiv 1 \pmod{k}$. Therefore, we have
		\[ |a| = |G \colon \langle b \rangle| \cdot |\langle a \rangle \cap \langle b \rangle|
			= m \cdot \frac{2km}{\gcd(2 k m, 2 \ell m)}
			= m \cdot \frac{2km}{2m}
			= km
			. \]
	Hence, $|G : \langle a \rangle| = |G|/|a| = m(2km)/(km) = 2m$. Also, since $ma = 2 \ell m b$ and $4 \ell^2 \equiv \pm 1 \pmod{k}$ (and $|b| = 2km$), we have
		\[ 4 \ell m a = (4\ell)(2\ell m b) = 4 \ell^2 (2m b) = \pm 2mb . \]
So this yields the graph $\Qd(2m, km, \pm 4 \ell m)$ of~\fullref{Qd}{kl} (even if $m = 1$).

However, this graph has twin vertices (and is therefore trivially stable) if (and only if) $2m = 2$ and $4 \ell m \equiv \pm 2 \pmod{km}$ \csee{GridTriv}. This situation is ruled out by assuming (at the end of \fullref{Qd}{kl}) that either $m > 1$ or $2\ell \not\equiv \pm 1 \pmod{k}$.
	
	\item[\pref{val4-2a=2b}] This is trivially unstable \csee{val4triv}.
	
	\item[\pref{val4-MoebiusPrism}] The generating sets arising here are not of the form $\{\pm a, \pm b\}$, so \cref{GridIsCayley} cannot be applied. These Cayley graphs are therefore not needed to find all of the toroidal grids.
	\qedhere
	\end{itemize}
\end{proof}

\begin{rem} \label{val4triv}
It is easy to determine whether a particular Cayley graph listed in \cref{val4} is trivially unstable. First, note that the graph is assumed to be connected (i.e., it is assumed that $S$ generates~$G$). \Cref{val4pf-triv} of the proof shows that the examples with twin vertices are precisely those in~\pref{val4-2a=2b}. So all that remains is to determine which of them are bipartite (which is usually answered by \fullcref{GridTriv}{QdBip}).
	\begin{itemize}
	\item[\fullref{val4}{int4}] This is bipartite if and only if $|a|/4 + |b|/4$ is even. It has twin vertices if and only if $|a| = |b| = 8$ (in which case, it is also bipartite).
	\item[\fullref{val4}{square}] This is bipartite if and only if $m$ is even. It has twin vertices if and only if $m = 1$ and $2\ell \equiv \pm 1 \pmod{k}$.
	\item[\fullref{val4}{2a=2b}] As mentioned above, this graph has twin vertices, and is therefore trivially unstable. (For completeness, we observe that it is bipartite if and only if $|a|$ and $|b|$ are even.)
	\item[\fullref{val4}{MoebiusPrism}] This is bipartite if and only if $n$ is odd.
	(It never has twin vertices.)
	\end{itemize}
\end{rem}

To avoid cluttering the main part of the proof of \cref{val4}, we present one direction of the argument in the following \lcnamecref{val4unstable}. It is mostly (or entirely?) known: the instability of the graphs in \pref{val4-int4} and~\pref{val4-square} was proved by S.\,Wilson \cite[\S A.4.1]{Wilson}, and the rest is very easy.
However, Wilson gave only a one-sentence sketch of his proofs, so we will provide a fairly complete argument for every case.

\begin{lem}[cf.\ S.\,Wilson {\cite[\S A.4.1]{Wilson}}] \label{val4unstable}
All of the graphs listed in \cref{val4} are unstable.
\end{lem}

\begin{proof}
We consider each part of the statement of the \lcnamecref{val4} individually. We may assume each Cayley graph is not bipartite (for otherwise it is trivially unstable).

\pref{val4-int4} (S.\,Wilson \cite[Thm.~Q.1, p.~380]{Wilson})
Let $m = |G : \langle b \rangle|$ and $n = |b|$, and choose $r \in \ZZ$, such that $ma = rb$. 
Also, let $\invol$ be the element of order~$2$ in $\langle a \rangle \cap \langle b \rangle$. Since $|b|$ is divisible by~$4$, we know that $|b + \invol| = |b| = n$, so $|G : \langle b + \invol \rangle| = |G : \langle b \rangle| = m$. 
Also note that, since $|\langle a \rangle \cap \langle b \rangle| = 4$, we have $|rb| = 4$, so $-rb = rb + \invol = rb + (m + r)\invol$ (since $m + r$ is odd, because the Cayley graph is not bipartite). Therefore
	\[ m(-a + \invol) = -ma + m\invol = -rb + m\invol = rb + (m + r)\invol + m\invol = r(b + \invol) . \]
Therefore, there is an automorphism~$\varphi$ of~$G$, such that $\varphi(a) = -a + \invol$ and $\varphi(b) = b + \invol$. Then $\varphi$ is an isomorphism from $\Cay(G; S)$ to $\Cay(G; S + \invol)$. This implies that $\Cay(G;S)$ is unstable \csee{IsoS+z}.

\pref{val4-square} (S.\,Wilson \cite[Thm.~Q.2, p.~381]{Wilson})
Let $n = 2km = |b|$ and $r = 2\ell m$, so $ma = rb$.
Also, let $\invol$ be the element of order~$2$ in~$\langle b \rangle$. 
Note that:
	\begin{itemize}
	\item $m$ is odd, because $X$ is bipartite and $r = 2\ell m$ is even,
	and
	\item $\gcd(2\ell, k) = 1$, because $4\ell^2 \equiv 1 \pmod{k}$ (so $k$ is odd).
	\end{itemize}
Then, since $|b| = 2km$ and $km$ is odd, we see that $|b + \invol| = km$. Also (using the fact that $\gcd(\ell, k) = 1$), we have
	\[ |a| 
		= m \cdot \frac{|b|}{\gcd(r, |b|)}
	 	= m \cdot \frac{2 k m}{\gcd(2\ell m, 2km)}
		= m \cdot  k . \]
Since $mk$ is odd, this implies $|a + \invol| = 2 mk = |b|$.

Also note that, since $4 \ell^2 = p k \pm 1$ for some $p \in \ZZ$ (and $p$ must be odd), we have 
	\[ 4 \ell^2 b = (p k \pm 1) b = pkb \pm b = p\invol \pm b = \invol \pm b. \]
Therefore
	\[ r (a + \invol)
	= 2 \ell m(a + \invol) 
	= 2 \ell r b + 0
	 = 2 \ell (2 \ell m b) 
	= 4\ell^2 m b
	= m(\invol \pm b)
	= \pm m(b + \invol)
	, \]
Hence, there is an automorphism~$\varphi$ of~$G$, such that $\varphi(b) = a + \invol$ and either $\varphi(a) = b + \invol$ or $\varphi(a) = -b + \invol$. In either case, $\varphi$ is an isomorphism from $\Cay(G; S)$ to $\Cay(G; S + \invol)$. This implies that $\Cay(G;S)$ is unstable \csee{IsoS+z}.

\pref{val4-2a=2b} $\Cay(G;S)$ has twin vertices \csee{val4triv}, so it is trivially unstable.

\pref{val4-MoebiusPrism}
Since the Cayley graph is not bipartite, we know that $n$ is even.
By \cref{B(XboxBip)}, we have
	\[ BX = B(M_{2n} \cartprod K_2)
		\iso (BM_{2n}) \cartprod K_2 
		\iso (C_{2n} \cartprod K_2) \cartprod K_2
		\iso C_{2n} \cartprod C_4
		. \]
So $|{\Aut BX}| \ge |{\Aut C_{2n}}| \cdot |{\Aut C_{4}}| = 4n \cdot 8 = 32 n$. 

If $n \ge 4$, then
	\[ |{\Aut X} | = |{\Aut (M_{2n}} \cartprod K_2) | = 2 |{\Aut M_{2n}}| = 8n < \frac{1}{2} |{\Aut BX}| , \]
so $X$ is unstable.

For the special case where $n = 2$, we have
 $X \iso K_4 \cartprod K_2$, so .
	\[ |{\Aut X}| = |{\Aut K_4}| \cdot |{\Aut K_2}| = 4! \cdot 2 . \]
However,
	\[ BX \iso (BK_4) \cartprod K_2 
	\iso (K_2 \cartprod K_2 \cartprod K_2 ) \cartprod K_2 , \]
so $|{\Aut BX}| = 4! \cdot 2^4 \gg 2 \, |{\Aut X}|$. Therefore $X$ is unstable.
\end{proof}

\begin{proof}[\bf Proof of \cref{val4}]
($\Leftarrow$) See \cref{val4unstable}.

\medbreak

($\Rightarrow$) Let $X = \Cay(G; S)$, and assume that $X$ is connected and unstable, but not bipartite. We will show that $X$ is in the list. 

\refstepcounter{caseholder} 

\begin{case} \label{val4pf-triv}
Assume $X$ is trivially unstable.
\end{case}
Since $X$ is assumed to be connected and nonbipartite, it must have twin vertices. Therefore, $S$ is a union of cosets of some subgroup~$\langle \twin \rangle$ of prime order \csee{twiniff}. Since $X$ has valency~$4$, we know that $|S| = 4$, so we must have $|\twin| = 2$ (since $|\twin|$ is a prime number that divides~$|S|$).

\begin{subcase}
Assume $S = \{\pm a, \pm b\}$, where $|a|, |b| > 2$.
\end{subcase}
We may assume $a + \twin \in \{-a, b\}$ (perhaps after replacing~$b$ with its negative). 
	\begin{itemize}
    	\item If $a + \twin = -a$, then $2a = \twin$ (and $-a + \twin = a$). Also, $b + \twin \notin \{\pm a\}$, so we must have $b + \twin = -b$, which implies $2b = \twin$. Therefore $2a = \twin = 2b$, so \pref{val4-2a=2b} is satisfied.
	\item If $a + \twin = b$, then $2b = 2(a + \twin) = 2a + 2\twin = 2a + 0 = 2a$, so \pref{val4-2a=2b} is satisfied.
	\end{itemize}

\begin{subcase}
Assume $S = \{\pm a, b, c\}$, where $|a| > 2$ and $|b| = |c| = 2$.
\end{subcase}
Since $b + \twin \in S$ and $2(b + \twin) = 2b + 2\twin = 0 + 0 = 0$, we must have $b + \twin = c$ (and hence $c + \twin = b$). So $a + \twin = -a$, which implies $\twin = 2a$ (and $|a| = 4$). Now, since $X$ is not bipartite, there exist $p,q,r \in \ZZ$, such that $pa + qb + rc = 0$ and $p + q + r$ is odd. Then
	\[ 0 = pa + qb + rc = pa + qb + r(b + 2a) \equiv (q + r) b \pmod{a} . \]
If $q + r$ is odd, this implies $b \in \langle a \rangle$, so $b = \twin$ (since $\twin$ is the unique element of order~$2$ in~$\langle a \rangle$. But then $c = b + \twin = \twin + \twin = 0$, which contradicts the fact that $|c| = 2$.

So $q + r$ is even. Therefore $p$ is odd, so $pa = \pm a$ (since $|a| = 4$). Then $\pm a = -(q b + rc) \in \langle b, c \rangle \iso \ZZ_2 \times \ZZ_2$. This is impossible, since $|a| = 4$.

\begin{subcase}
Assume $S = \{ a, b, c, d\}$, where $|a| = |b| = |c| = |d| = 2$.
\end{subcase}
We may assume, without loss of generality, that $a + \twin = b$ and $c + \twin = d$, and also, since $X$ is not bipartite, that $a + b + c = 0$. But then 
	\[ c = a + b = a + (a + \twin) = 0 + \twin = \twin, 
	\quad \text{so} \quad
	 d = c + \twin = c + c = 0, \]
which contradicts the fact that $|d| = 2$.

\begin{pfassump} \label{val4pf-nontriv}
In the remaining cases of the proof, we assume that $X$ is nontrivially unstable.
\end{pfassump}

\begin{case} \label{val4-cyclic}
Assume $G$ is cyclic.
\end{case}
We see from \cref{val4circulant} that $X$ is listed in either \pref{val4-int4} or \pref{val4-square} (with $m = 1$).

\begin{case}
Assume that $S$ contains at least one element of order~$2$.
\end{case}
If every element of~$S$ has order~$2$, then (since $X$ is not bipartite) it is not difficult to see that $X$ is the Cayley graph $K_4 \cartprod K_2$ that is listed in~\pref{val4-MoebiusPrism-n=4}. 

Therefore, we may assume that $S$ contains precisely two elements of order~$2$, so we may write $S = \{a,b,\pm c\}$, where $|a| = |b| = 2$ and $|c| \ge 3$. Since $\langle c \rangle$ has at most one element of order~$2$, we have $| \langle a, b \rangle \cap \langle c \rangle| \in \{1,2\}$. Let $n = |c|$.

\begin{subcase}
Assume $| \langle a, b \rangle \cap \langle c \rangle| = 1$.
\end{subcase}
Then 
	\[ X \iso C_n \cartprod C_4 \iso \Cay \bigl( \ZZ_n \times \ZZ_4; \pm(1,0), \pm(0,1) \bigr) . \]
This Cayley graph has no elements of order~$2$ in the generating set, so it is considered in a later case. Also note that $n$ must be odd, since $X$ is not bipartite. Then it is not difficult to see that this Cayley graph is not listed in any of the parts of the statement of the \lcnamecref{val4}, so it is stable.

\begin{subcase}
Assume $| \langle a, b \rangle \cap \langle c \rangle| = 2$, but $\langle c \rangle \cap \{a,b\} = \emptyset$.
\end{subcase}
Then $X$ is a prism with $2n$ vertices, plus an edge from each vertex to its antipodal vertex. (Note that $n$ must be even, since $\langle c \rangle$ has a subgroup of order~$2$.) Therefore, it is not difficult to see that $X \iso \Cay(G'; \pm x, \pm y)$, where  
	\[ G' = \langle\, x,y \mid 4x = ny = 0, \ 2x = (n/2)y , \ x + y = y + x \, \rangle . \]
The generating set of this abelian Cayley graph has no elements of order~$2$, so the Cayley graph is considered in a later case. Since $X$ is not bipartite, $n/2$ must be odd. Therefore, it is not difficult to see that this Cayley graph is not listed in any of the parts of the statement of the \lcnamecref{val4}, so it is stable.

\begin{subcase}
Assume $\langle a, b \rangle \cap \langle c \rangle = \langle a \rangle$.
\end{subcase}
Then $X \iso M_n \cartprod K_2$ is listed in~\pref{val4-MoebiusPrism}.

\begin{pfassump}
In the remaining cases of the proof, we assume that $S$ has no elements of order~$2$. Therefore, we may write 
	\[ \text{$S = \{\pm a, \pm b\}$, where $|a|, |b| > 2$.} \]
\end{pfassump}

\begin{case} \label{val4pf-aut}
Assume there is a group automorphism~$\alpha$ of $G \times \ZZ_2$, such that $\alpha$ is an automorphism of $BX$, and $\alpha(0,1) \neq (0,1)$.
\end{case}
Since $G \times \{0\}$ and $G \times \{1\}$ are the bipartition sets of $BX$, we know that each of these sets is $\alpha$-invariant.
Therefore
	\begin{itemize}
	\item $\alpha(g,0) = \bigl( \varphi(g), 0 \bigr)$, for some automorphism~$\varphi$ of~$G$,
	and
	\item $\alpha(0,1) = (\invol, 1)$, for some element~$\invol$ of order~$2$.
	\end{itemize}
Since $\alpha$ is an automorphism of~$BX$, we must have $\alpha \bigl(S \times \{1\} \bigr) = S \times \{1\}$, so $\varphi(S) = S + \invol$. Therefore $\varphi$ is an isomorphism from $\Cay(G; S)$ to $\Cay(G; S + \invol)$.

We may assume (by interchanging $a$ and~$b$, if necessary) that $|b|$ is divisible by (at least) the largest power of~$2$ that divides~$|a|$. By \cref{odd} (and \cref{val4pf-nontriv}), this implies that
	\[ \text{$|b|$ is even.} \]
Let 
	\[ m = |G : \langle b \rangle| . \]
Then $ma \in \langle b \rangle$, so we may choose $r \in \{0,1,\ldots,|b| - 1\}$, such that 
	\[ ma + rb = 0 . \]
Since $X$ is not bipartite, we know that 
	\[ \text{$m + r$ is odd.} \]

\begin{subcase}
Assume $\varphi(b) \in \{\pm b + \invol\}$.
\end{subcase}
Since $\varphi$ is a homomorphism (and $|\invol| = 2$), this implies that $\varphi \bigl( \{ \pm b \} \bigr) = \{ \pm b + \invol \}$. Then, since $\varphi$ is a bijection from $S$ to $S + \invol$, we must have $\varphi(a) = \epsilon a + \invol$, for some $\epsilon \in \{\pm 1\}$. 
We may assume, without loss of generality, that $\varphi(b) = b + \invol$ (by composing with the automorphism $x \mapsto -x$, if necessary).
Then
	\begin{align*}
	0
	&= \varphi(ma + rb)
		&& \text{($ma + rb = 0$)}
	\\&= m \, \varphi(a) + r \, \varphi(b)
		&& \text{($\varphi$ is a group automorphism)}
	\\&= m \, (\epsilon a + \invol) + r \, (b + \invol)
	\\&= \epsilon m a + r b + \invol
		&& \text{($|\invol| = 2$ and $m + r$ is odd)}
	. \end{align*}
If $\epsilon = 1$, then $\epsilon ma + r b = ma + r b = 0$, so $\invol = 0$, which contradicts the fact that $|\invol| = 2$.

Therefore, we must have $\epsilon = -1$, so $-ma + rb = \invol$. Subtracting this from the equation $ma + rb = 0$, we conclude that $2ma = \invol$ has order~$2$, so $ma$ has order~$4$. Thus, the Cayley graph is listed in~\pref{val4-int4}.

\begin{subcase}
Assume $\varphi(b) \in \{\pm a + \invol\}$.
\end{subcase}
We may assume, without loss of generality, that $\varphi(b) = a + \invol$ (by composing with the automorphism $x \mapsto -x$, if necessary). We have $\varphi(a) = \epsilon b + \invol$, for some $\epsilon \in \{\pm 1\}$.

Note that $|b| = |\varphi(b)| = |a + \invol|$.
Therefore, either $|b| = 2 |a|$ (and $|a|$ is odd) or $|b| = |a|$. Hence, $\gcd\bigl( |b| , r \bigr) \in \{m, 2m\}$. If $m$ is even, this implies that $r$~is also even, which contradicts the fact that $m + r$ is odd. Therefore
	$m$ is odd, so $r$ is even. 
Hence, we must have $\gcd\bigl( |b| , r \bigr) = 2m$, so 
we may write 
 	\[ \text{$|b| = 2 k m$ \ and \ $r = 2 \ell m$, \quad for some $k,\ell \in \ZZ$.} \]

We have
	\begin{align*}
	0
	&= \varphi(ma + rb)
		&& \text{($ma + rb = 0$)}
	\\&= m \, \varphi(a) + r \, \varphi(b)
		&& \text{($\varphi$ is a group automorphism)}
	\\&= m \, (\epsilon b + \invol) + r \, (a + \invol)
	\\&= ra + \epsilon m b + \invol
		&& \text{($|\invol| = 2$ and $m + r$ is odd)}
	\\&= 2 \ell m a + \epsilon m b + \invol
		&& \text{(definition of~$\ell$)}
	. \end{align*}
We also have
	\[ 2 \ell m a + 4 \ell^2 m b = 2\ell(ma + rb) = 2\ell(0) = 0 , \]
so, by subtracting these two equations, we conclude that $(4 \ell^2 - \epsilon)mb = \invol$. Since $|b| = 2km$, and $|\invol| = 2$, then $4 \ell^2 \equiv \epsilon \pmod{k}$. Thus, the Cayley graph is listed in~\pref{val4-square}.

\begin{pfremark}
The remaining cases are copied almost verbatim from the analogous cases in \cite[proof of Thm.~4.3]{HMMlow}.
\end{pfremark}

\begin{case} \label{val4-neq}
Assume $2s \neq 2t$, for all $s,t \in S$, such that $s \neq t$.
\end{case}
By \cref{StableIffStabilizer}, there is an automorphism~$\alpha$ of~$BX$ that fixes $(0,0)$, but does not fix $(0,1)$.
We may assume $\alpha$ is not a group automorphism, for otherwise \cref{val4pf-aut} applies. 
Therefore, \cref{4cyclenormal} implies there exist $s,t,u,v\in S$ such that  $s + t = u + v \neq 0$ and $\{s,t\} \neq \{u,v\}$. From the assumption of this \lcnamecref{val4-neq},  we see that this implies $3b = \pm a$ (perhaps after interchanging $a$ with~$b$). Then 
	\[ \text{$G = \langle a, b \rangle = \langle 3b, b \rangle = \langle b \rangle$ is cyclic} , \]
so \cref{val4-cyclic} applies.

\begin{case} \label{val4-remain}
The remaining case.
\end{case}
Since \cref{val4-neq} does not apply, we have $2s = 2t$, for some $s,t \in S$, such that $s \neq t$. 

\begin{subcase}
Assume that $t = -s$. 
\end{subcase}
Then $|s| = 4$. Therefore, if we assume, without loss of generality, that $s = a$, then we have $i \coloneqq | \langle a \rangle \cap \langle b \rangle | \in \{1,2,4\}$. In all cases, we will show that $G$ is cyclic, so \cref{val4-cyclic} applies.

If $i = 1$, then $G = \langle a \rangle \times \langle b \rangle$. Since $X$ is not bipartite, this implies $|b|$ is odd, so $\gcd \bigl( |a|, |b| \bigr) = 1$. Therefore $G$ is cyclic.

If $i = 2$, then there is some $k \in \ZZ$, such that $a^2 = b^k$. Since $X$ is not bipartite, we know $k$ is odd. So $\langle a^2 \rangle$ has odd index in $\langle b \rangle$. This implies that $\langle a \rangle$ has odd index in~$G$ and is therefore a Sylow $2$-subgroup. So the Sylow $2$-subgroup of~$G$ is cyclic.  All of the other Sylow subgroups of~$G$ are contained in $\langle b \rangle$, and are therefore also cyclic.  So $G$ is an abelian group whose Sylow subgroups are cyclic. Therefore $G$ is cyclic.

If $i = 4$, then $a \in \langle b \rangle$, so $G = \langle a,b \rangle = \langle b \rangle$ is cyclic.

\begin{subcase}
Assume that $t \neq -s$. 
\end{subcase}
Therefore, we may assume $s = a$ and $t = b$, so $2a = 2b$. If we let $\invol = b - a$, this implies that $2\invol = 0$, so $\invol = -\invol$. Then $a = b + \invol$ and $b = a - \invol = a + \invol$, so $S = S + \invol$, which contradicts the assumption that $\Cay(G: S)$ is nontrivially unstable (and therefore has no twin vertices).
\end{proof}

\section{Some unstable abelian Cayley graphs of valency 6} \label{val6sect}

In this \lcnamecref{val6sect}, we prove the following \lcnamecref{val6}, which implies \cref{Tr}.
(Although the title of this \lcnamecref{val6sect} specifies ``valency~6\rlap,'' the \lcnamecref{val6} also applies to some graphs of smaller valency, because $a$, $b$, and/or~$c$ may have order~$2$.)

\begin{thm} \label{val6}
Let $\{a, b, c\}$ be a generating set of a finite abelian group~$G$, such that 
	\[ \text{$a + b + c = 0$
	\quad and \quad
	the sets $\{\pm a\}$, $\{\pm b\}$, $\{\pm c\}$ are distinct.} \]
The Cayley graph $X = \Cay(G; \pm a, \pm b, \pm c)$ is unstable if and only if one of the following conditions is satisfied \textup(perhaps after permuting $a$, $b$, and~$c$\kern1pt\textup{):}
	\begin{enumerate}

	\item \label{val6-4}
  	$|a| = 4$ and $|G|$ is divisible by~$8$.

	\item \label{val6-2a=2b}
	$2a = 2b$ and $|G|$ is divisible by~$8$.
	
	\item \label{val6-8}
	$|a| = 8$ and $b = 3a$.

	\item \label{val6-cyclic12}
	$|a| = 12$ and $b = 4 a$.
	
	\item \label{val6-9}
	$|a| = |b| = 3$.

	\end{enumerate}
\end{thm}

Before proving this \lcnamecref{val6}, let us show that it implies \cref{Tr}.

\begin{proof}[\bf Proof of \cref{Tr}]
As in the proof of \cref{Qd}, we use \cref{GridIsCayley}. We will show that the graphs in \cref{Tr} are precisely those that arise from \cref{val6} (and are not the two trivially unstable graphs in \fullcref{GridTriv}{TrTwin}), and satisfy the additional assumption that $m,n \ge 2$ (where $n = |b|$ and $m = |G : \langle b \rangle|$).
To do this, we consider each part of the statement of \cref{val6} individually. We also consider appropriate permutations of $a$, $b$, and~$c$. This is made easier by the observation that \cref{abyc} determines the result of interchanging $a$ with~$c$.

\pref{val6-4}
Since $|a| = 4$, we must have $m \in \{1, 2,4\}$.
Since $|G|$ is divisible by~$8$, but $|a|$ is not, we know that $n$ and~$r$ are even.
	\begin{itemize}
	\item If $m = 4$, then $r = 0$ (because $ma = 4a = 0$). This yields the graph $\Tr(4, 2k, 0)$ of~\fullref{Tr}{4/4k/0}. By \cref{abyc}, this is isomorphic to $\Tr(4, 2k, 4)$.
	
	By using $a$ in the role of~$b$ in \cref{GridIsCayley}, we obtain $\Tr(2k, 4, 0)$, which is also listed in~\fullref{Tr}{4/4k/0}. If $k$ is even, then applying \cref{abyc} does not give anything new. However, if $k$ is odd, then this graph is isomorphic to $\Tr(2k, 4, 2)$, as mentioned at the end of the statement of \cref{Tr}.
	
	\item If $m = 2$, then $r = n/2$ (because $ma = 2a$ has order~$2$). Since $r$ is even, this yields the graph $\Tr(2, 4k, 2k)$ of~\fullref{Tr}{2k/4/2}.
	By \cref{abyc}, this is isomorphic to $\Tr(2, 4k, 2 - 2k) = \Tr(2, 4k, 2k + 2)$.

	By using $a$ in the role of~$b$ in \cref{GridIsCayley}, we obtain $\Tr(2k, 4, 2)$, which is also listed in~\fullref{Tr}{2k/4/2}. If $k$ is even, then, as above, applying \cref{abyc} does not give anything new. However, if $k$ is odd, then this graph is isomorphic to $\Tr(2k, 4, 0)$, as mentioned at the end of the statement of \cref{Tr}.
	
	\item If $m = 1$, then we cannot directly apply \cref{GridIsCayley}, because the definition of $\Tr(m,n,r)$ requires $m > 1$. 
	
	However, we may use $a$ in the role of~$b$. This yields $\Tr(2k, 4, r)$, and we have $r \in \{\pm1\}$ because $\langle b \rangle = G$ (since $m = 1$). Therefore, the graph is listed in \fullref{Tr}{2k/4/pm1}. If $k$ is even, then the two graphs are isomorphic (by \cref{abyc}), as mentioned at the end of the statement of \cref{Tr}. The graph $\Tr(2, 4, 1)$ is trivially unstable \csee{GridTriv}, so we require $k > 1$ in $\Tr(2k, 4, 1)$.
	\end{itemize}

\pref{val6-2a=2b}
Write $|G| = 8k$.
	\begin{itemize}
	\item Suppose, for the moment, that $\langle b \rangle = G$. Since $2a = 2b$, but $a \neq b$, we must have $a = (4k + 1) b$. Then $a$ and~$b$ both generate~$G$, so neither can play the role of~$b$ in \cref{GridIsCayley}. However, we may let $c$ play this role. Note that $c = -(a + b) = -(4k + 2) b = (4k - 2)b$, so $|G : \langle c \rangle| = 2$ and
	\[ 2a = 2b = (2k - 1)(4k - 2)b = (2k - 1) c , \]
so this yields the graph $\Tr \bigl( 2, 4k, -(2k - 1) \bigr) = \Tr(2, 4k, 2k + 1)$, which is listed in~\fullref{Tr}{2/4k/2k+1}. Applying \cref{abyc} to this graph does not yield anything new.
	
	\item We may now assume $\langle b \rangle \neq G$. Then, since $2a = 2b$, we have $m = r = 2$ and $n = |G|/m = 8k/2 = 4k$. So this yields the graph $\Tr(2, 4k, -2)$ of~\fullref{Tr}{4/2k/2}. By \cref{abyc}, this is isomorphic to $\Tr(2, 4k, 4)$.

Note that $4a = 2a + 2a = 2a + 2b = 2(a+b) = -2c$. Therefore, by using $c$ in the role of~$b$, we obtain the graph $\Tr(4, 2k, 2)$ of~\fullref{Tr}{4/2k/2}. Applying \cref{abyc} to this graph does not give us anything new.
	\end{itemize}

\pref{val6-8} Since $\langle a \rangle = \langle b \rangle = G$, neither~$a$ nor~$b$ can play the role of~$b$ in \cref{GridIsCayley}. Letting $c$ play the role of~$b$ yields the graph $\Tr(4, 2, 1)$ (because $c = -(a + b) = -4a$ has order~$2$). This is listed in~\fullref{Tr}{4/23/pm1}. Applying \cref{abyc} to this graph does not yield anything new.

\pref{val6-cyclic12}
Since $G = \langle a \rangle$, we have $|G| = |a| = 12$. Then $n = |b| = 3$ and $m = |G|/|b| = 4$. Also, $ma = 4a = b$, so $r = 1$. Therefore, we have the graph $\Tr(4, 3, -1)$ of~\fullref{Tr}{4/23/pm1}. Applying \cref{abyc} to this graph does not yield anything new.  

Since $\langle a \rangle = \langle c \rangle = G$, neither $a$ nor~$c$ can play the role of~$b$ in \cref{GridIsCayley}.

\pref{val6-9} This graph is trivially unstable \fullcsee{val6triv}{3}.
\end{proof}

To shorten the main argument, we establish three minor results that deal with parts of the proof of \cref{val6}.

\begin{lem} \label{val6triv}
Let $\{a, b, c\}$ be a generating set of a finite abelian group~$G$, such that 
	\[ \text{$a + b + c = 0$
	\quad and \quad
	the sets $\{\pm a\}$, $\{\pm b\}$, $\{\pm c\}$ are distinct.} \]
The Cayley graph $X = \Cay(G; \pm a, \pm b, \pm c)$ has twin vertices \textup(or, equivalently, is trivially unstable\textup) if and only if \textup(perhaps after permuting $a$, $b$, and~$c$\textup) either:
	\begin{enumerate}
	\item \label{val6triv-8}
	$|a| = 8$ and $b = 2a$,
	or
	\item \label{val6triv-3}
	$|a| = |b| = 3$.
	\end{enumerate}
In each case, the Cayley graph is listed in \cref{val6}.
\end{lem}

\begin{proof}
($\Leftarrow$) Up to a group isomorphism, $X$ is either $\Cay(\ZZ_8; S_8)$ or $\Cay(\ZZ_3 \times \ZZ_3; S_3)$, where 
	\[ \text{$S_8 = \{\pm 1, \pm 2, \pm 3\}$
	\quad and \quad
	$S_3 = \{\pm(1,0), \pm(0,1), \pm(1,1)\}$.}
	\]
We have $S_8 = S_8 + 4$ and $S_3 = S_3 + (-1,1)$, so both Cayley graphs have twin vertices, and are therefore trivially unstable. 

Since $2$ is an element of order~$4$ in~$\ZZ_8$, the first Cayley graph is listed in~\fullref{val6}{4}. The second is listed in~\fullref{val6}{9}.

\medbreak

($\Rightarrow$) 
Let $S = \{\pm a, \pm b, \pm c\}$. Since $X$ has twin vertices, we see from \cref{twiniff} that $S$ is a union of cosets of some subgroup~$\langle \twin \rangle$ of prime order. Since $|S| \le 6$, we must have $|z| \in \{2,3,5\}$. 

\refstepcounter{caseholder}  

\begin{case}
Assume $|\twin| = 2$.
\end{case}
Then $-s + \twin = -(s + \twin)$ for all $s \in S$, so the permutation $x \mapsto x + \twin$ induces a well-defined action on $\bigl\{ \{\pm a\}, \{\pm b\}, \{\pm c\} \bigr\}$. Since the permutation has order~$2$, this implies $s + \twin = -s$ for some $s \in S$. Assume without loss of generality that $s = a$, so $|a| = 4$ and $\twin = 2a$. 

If $b + \twin = -b$, then we also have $c + \twin = -c$, so 
		\[ 0 = -(a + b + c) = (a + \twin) + (b + \twin) + (c + \twin) = (a + b + c) + \twin = 0 + \twin = \twin , \]
which contradicts the fact that $|\twin| = 2$.

Therefore, we must have $b + \twin = \pm c$, so $a + b \pm (b + \twin) = 0$.
	\begin{itemize}
	\item For the minus sign, we have $0 = a + b - (b + \twin) = a + \twin$. This contradicts the fact that $|a| = 4 \neq |\twin|$.
	\item For the plus sign, we have $0 = a + 2b + \twin = -a + 2b$, so $a = 2b$. Since $|a| = 4$, this implies $|b| = 8$, so $X$ is the Cayley graph in~\pref{val6triv-8}.
	\end{itemize}

\begin{case}
Assume $|\twin| = 3$.
\end{case}
This implies $|S| = 6$, so no element of~$S$ has order~$2$. Therefore, if $C$ is any coset of~$\langle \twin \rangle$, then $C \neq -C$ (since $|C|$ is odd). Since the cosets form a partition, we conclude that $C \cap -C = \emptyset$. Hence, we may assume that $C$ contains $a$, $b$, and~$\pm c$. Then 
	\[ 0 = a + (a + \twin) \pm (a + 2\twin) , \]
so either $3a = 0$ or $a = \twin$. However, $a \neq \twin$, since $0 \notin S$. Therefore $3a = 0$, which means $|a| = 3$. Since $\twin$ also has order~$3$, $X$ is the Cayley graph in~\pref{val6triv-3}.

\begin{case}
Assume $|\twin| = 5$.
\end{case}
Then $|S| = 5$, so some element~$s$ of~$S$ has order 2. Since $G = \langle S \rangle = \langle s, \twin \rangle$, this implies $|G| = 10$. More precisely, up to a group isomorphism, we have $X = \Cay(\ZZ_{10}; \pm 1, \pm 3, 5)$. However, it is not possible to choose representatives of $\{\pm 1\}$, $\{\pm 3\}$, and~$\{5\}$ whose sum is~$0$, so this case is not possible.
\end{proof}

\begin{lem} \label{val6unstable}
All of the Cayley graphs listed in the statement of \cref{val6} are unstable.
\end{lem}

\begin{proof}
\pref{val6-4} (Wilson \cite[Thms.\ T.2 and~T.3, p.~381]{Wilson})
Let $\invol = 2a$, and define $\varphi \colon G \to G$ by $\varphi(pa + qb) = pa + q(b + \invol)$. Since $|a| = 4$, we have $|\invol| = 2$, so $\varphi$ is well-defined. (Since $|G|$ is divisible by~$8$, but $|a|$ is not, and $\langle a, b \rangle = G$, we know that $|\langle b \rangle : \langle a \rangle  \cap \langle b \rangle |$ is even. Therefore, if $p_1 a + q_1 b = p_2 a + q_2 b$, then $q_1 \equiv q_2 \pmod{2}$, so $p_1 a + q_1 (b + \invol) = p_2 a + q_2 (b + \invol)$.) Then it is easy to see that $\varphi$ is an automorphism of~$G$. Also, we have 
	\[ \varphi(a) = a = -3a = -a + 2a = -a + \invol , \]
so $\varphi(S) = S + \invol$. Hence, $\varphi$ is an isomorphism from $\Cay(G; S)$ to $\Cay(G; S + \invol)$. This implies $\Cay(G; S)$ is unstable \csee{IsoS+z}.

\pref{val6-2a=2b} (Wilson \cite[Thm.~T.1, p.~381]{Wilson})
There is an automorphism~$\varphi$ of~$G$ that interchanges $a$ and~$b$ (and fixes~$c$). Then $\varphi(S) = S$, so $\varphi$ is an automorphism of $\Cay(G; S)$. Also note that $\varphi$ fixes each element of the index-$2$ subgroup $\langle 2a, c \rangle$.
	\begin{itemize}
	\item If $\langle a \rangle \neq G$, then the subgraph induced by the set of un-fixed vertices consists of two cycles of length~$|c|$, and these two cycles are interchanged by~$\varphi$. Since $|c| = |a|/2 = |G|/4$ is even, these cycles are bipartite. Therefore $\Cay(G; S)$ is unstable by \cref{LargelyFixing}.	
	\item If $\langle a \rangle = G$, then (since $|G|$ is divisible by~$8$) we may assume $G = \ZZ_{8k}$ and $a = 1$. Since $2a = 2b$, we know that $b - a$ has order~$2$, and is therefore equal to $4k$. Hence $S = \{\pm 1, 4k \pm 1, 4k \pm 2 \}$, so we see that $X$ is unstable by letting $a = 4k + 2$ and $b = 1$ in \fullcref{val6circulant}{C3order2}.
	\end{itemize}

\pref{val6-8} We have $X \iso \Cay(\ZZ_8; \pm 1, \pm 3, 4)$, which is unstable by \fullcref{val5circulant}{8}.

\pref{val6-cyclic12} Since $X \iso \Cay(\ZZ_{12}; \pm 1, \pm 4, \pm 5) = \Cay(\ZZ_{12}; \pm 1, \pm 4, \pm 7)$, we see that it is unstable from \fullcref{val6}{cyclic12} with $k = 3$ and $a = 4$.

\pref{val6-9} $X$ is trivially unstable by \fullcref{val6triv}{3}.
\end{proof}

\begin{lem} \label{val6abcCyclic}
Let $S = \{\pm 1, \pm b, \pm (b + 1) \}$, for some $b \in \ZZ_n$, such that $b \notin \{0, \pm 1, -2\}$ \textup(so the sets $\{\pm 1\}$, $\{\pm b\}$, $\{\pm (b + 1) \}$ are distinct\textup).
If\/ $\Cay(\ZZ_n; S)$ is unstable, then it is listed in \cref{val6}.
\end{lem}

\begin{proof}
Let $X = \Cay(\ZZ_n; S)$.
We may assume $X$ is nontrivially unstable, for otherwise \cref{val6triv} applies.
Also, since cyclic groups have no more than one element of order~$2$, we know that the valency of~$X$ is either $5$ or~$6$. 
Therefore, the Cayley graph $X$ is listed in either \cref{val5circulant} or \cref{val6circulant}.

\refstepcounter{caseholder} 

\begin{case}
Assume $X$ is listed in \cref{val5circulant}.
\end{case}
In \fullref{val5circulant}{12k}, one generator is odd, but the other two generators are even, so $a + b + c \not\equiv 0 \pmod{2}$. Hence, the equation $a + b + c = 0$ is not satisfied.

So $X$ must the graph in \fullref{val5circulant}{8}, which is listed in~\fullref{val6}{8}.

\begin{case} \label{val6abcCyclicpf-val6}
Assume $X$ is listed in \cref{val6circulant}.
\end{case}
We consider each of the seven lists of graphs individually.
	
	\smallbreak \textbf{\fullref{val6circulant}{C1Se2}}: The element $2k$ has order~$4$, so $X$ is listed in~\fullref{val6}{4}.
	
	\smallbreak \textbf{\fullref{val6circulant}{C1Se4}}: The generator $\pm a$ is odd, but $\pm b$ and $\pm b + 2k$ are even, so the equation $a + b + c = 0$ is not satisfied.
	
	\smallbreak \textbf{\fullref{val6circulant}{C2}}: For some choice of the signs, we must have $\pm a \pm (a + k) \pm (a - k) = 0$, so $3a \equiv 0 \pmod{k}$. Since $k$ is odd and we also have $a \equiv 0 \pmod{4}$, this implies $3a = 0$ in~$\ZZ_{4k}$. But since $1$ is in the generating set, we must also have $a - k = \pm 1$ (perhaps after replacing $a$ with $-a$). So $a \in \{k \pm 1\}$. Therefore $3k \pm 3 = 3a \equiv 0 \pmod{4k}$. From this (and the fact that $a \neq 0$), we conclude that $k = 3$. Since $k \pm 1 = a \equiv 0 \pmod{4}$, we see that $a = 4$, so $X$ is listed in~\fullref{val6}{cyclic12}.
	
	\smallbreak \textbf{\fullref{val6circulant}{C3order2}}: For some $c \in \{4k \pm b\}$, we have $\pm a + b + c = 0$. However, if $c = 4k - b$, then this implies $a = 4k$, which contradicts the fact that $|a|$ is divisible by~$4$. Therefore, we must have $c = 4k + b$. Then $2b = 2c$, so this Cayley graph is listed in~\fullref{val6}{2a=2b}.
	
	\smallbreak \textbf{\fullref{val6circulant}{C3order4}}: The condition $a + b + c = 0$ implies $a \equiv 0 \pmod{k}$, so all of the elements of~$S$ are divisible by~$k$. Since $1 \in S$, we conclude that $k = 1$, so $n = 8$. Since $a \equiv 0 \pmod{4}$, this implies $a = 4$. So $X$ is the Cayley graph in \fullref{val6}{8}. (Alternatively, we have $|a| = 2$, which contradicts the fact that the valency of~$X$ is assumed to be~$6$ in the current \lcnamecref{val6abcCyclicpf-val6}.)
	
	\smallbreak \textbf{\fullref{val6circulant}{C4}}: Since $m$ is odd, we know that $b$ and $mb + 2k$ have the same parity. Hence, the equation $a + b + c = 0$ implies that $a$ is even, so $a \neq \pm 1$. Therefore, we may assume $1 \in \{b , mb + 2k\}$. Since the assumptions imply $m^2 b \equiv \pm b$ (and $m$ is odd), we have $\pm b = m(mb + 2k) + 2k$, so there is no harm in assuming $b = 1$.

Since $m^2 \equiv 1 \pmod{4}$, we know that $m^2 + 1 \not\equiv 0 \pmod{4}$, so $(m^2 + 1)b = m^2 + 1 \not\equiv 0 \pmod{4k}$. 
Therefore, we must have $m^2 \equiv 1 \pmod{4k}$. 

For an appropriate choice of the sign (and perhaps replacing $a$ with its negative), we have $-a \pm b + (mb + 2k) = 0$, so
	\[ a = mb + 2k \pm b = m + 2k \pm 1. \]
However,
		\[ (m-1)(m + 2k + 1)
			= m^2 - 1 + (m-1)(2k)
			\equiv 1 - 1 + 0
			= 0
			\not\equiv 2k
			\equiv (m-1) a
			\pmod{4k}
			. \]
So $a \neq m + 2k + 1$. Hence, we must have
	\[ a = m + 2k - 1 . \]
Then, modulo~$4k$, we have
		\[ 2k \equiv (m-1) a
			= (m-1)(m + 2k - 1)
			= m^2 - 2m + 1+ (m-1)(2k)
			\equiv 1 - 2m + 1 + 0
			= -2(m - 1)
			. \]
This means $m \equiv k + 1 \pmod{2k}$, so $m$ is either $k + 1$ or $-k + 1$. 

Since $m$ is odd, this implies that $k$ is even, so $|\ZZ_{4k}|$ is divisible by~$8$. It also implies that $a$ is either $3k$ or~$k$. In either case, $a$ has order~$4$. So $X$ is listed in~\fullref{val6}{4}.

	\smallbreak \textbf{\fullref{val6circulant}{C4more}}: If $a = 1$, then $m = 4k + 1$, so $b$ and~$c$ must be even, which contradicts the equation $a + b + c = 0$. 
	
	Thus, we may assume $b = 1$. Then $m = 4k - 1$, so $a \equiv 2k \pmod{4k}$, which means $a \in \{\pm 2k\}$. Therefore $a$ has order~$4$, so the Cayley graph is listed in~\fullref{val6}{4}. 
\end{proof}

As final preparation for proof of \cref{val6}, let us recall some useful notation.

\begin{notation}[cf.\ {\cite[Defn.~2.3]{HMMaut}, \cite[Notn.~2.5]{KutnarEtAl}}]
Assume $X = \Cay(G; S)$ is an abelian Cayley graph.
	\noprelistbreak
	\begin{enumerate}
	\item For $g \in G$, let $\bx g = (g,1)$, so $BX = \Cay(G \times \ZZ_2; \bx S \,)$.
	\item For $s_1,\ldots, s_\ell \in S$, and a starting point $v \in V(BX)$, we use $(\bx{s_1},\bx{s_2},\ldots,\bx{s_\ell})$ to denote 
		\[ \text{the walk \ $v, \, v + \bx{s_1}, \, v + \bx{s_1} + \bx{s_2}, \, \ldots, \, v + \bx{s_1} + \bx{s_2} + \cdots + \bx{s_\ell}$ \ in $BX$}. \]
	\item $(\bx{s_1}^{m_1},\bx{s_2}^{m_2},\ldots,\bx{s_\ell}^{m_\ell})$ denotes the sequence consisting of 
	$m_1$ copies of $\bx{s_1}$, 
	followed by $m_2$ copies of $\bx{s_2}$, 
	followed by~\ldots,
	followed by $m_\ell$ copies of $\bx{s_\ell}$.
	\end{enumerate}
\end{notation}

\begin{proof}[\bf Proof of \cref{val6}]
($\Leftarrow$) See \cref{val6unstable}.

\medbreak

($\Rightarrow$)
Let $S = \{\pm a, \pm b, \pm c\}$, so $X = \Cay(G; S)$.
Assuming that $X$ is unstable, we will show that (at least) one of the listed conditions is satisfied.
Note that, since the equation $a + b + c  = 0$ is completely symmetric, we are free to permute the elements of $\{a,b,c\}$ arbitrarily.

By \cref{StableIffStabilizer}, we may let $\alpha$ be an automorphism of~$BX$ that fixes $(0,0)$, but does not fix $(0,1)$. Since $G \times \{0\}$ and $G \times \{1\}$ are the bipartition sets of $BX$, we know that each of these sets is $\alpha$-invariant. 

\refstepcounter{caseholder} 

\begin{case} \label{val6pf-odd}
Assume $|G|$ is odd.
\end{case}
By \cref{odd}, we know that $X$ is trivially unstable. 
Therefore, \cref{val6triv} applies.

\begin{case} \label{val6pf-cyclic}
Assume $G = \langle s \rangle$ for some $s \in S$.
\end{case}
See \cref{val6abcCyclic}.

\begin{case} \label{val6pf-order2}
Assume $|c| = 2$.
\end{case}
We may assume \cref{val6pf-cyclic} does not apply; therefore $c \notin \langle a \rangle$. If we let $n = |a|$, then
	\[ X \iso \Cay \bigl( \ZZ_n \times \ZZ_2; \pm(1,0), \pm(1,1), (0,1) \bigr) . \]
	\begin{itemize}
	\item If $n \ge 8$, then \cref{triangles} tells us that $X$ is stable.
	\item If $n \le 7$, and $n \neq 4$, then \cref{SmallZnxZ2} tells us that $X$ is stable.
	\item If $n = 4$, then $|a| = 4$ and $|G| = 8$, so condition~\pref{val6-4} is satisfied.
	\end{itemize}

\begin{pfassump}
Henceforth, we assume $|s| \ge 3$ for all $s \in S$ \textup(so $X$ has valency~$6$\textup), and $|G|$ is even.
\end{pfassump}

\begin{case} \label{val6pf-order4}
Assume $|s| = 4$, for some $s \in S$.
\end{case}
We may assume $|G| = 4k$, where $k$ is odd, for otherwise condition~\pref{val6-4} is satisfied. Then $G$ is cyclic (since $|a| = 4$ and $G/\langle a \rangle$ is generated by~$b$). We may also assume that $X$ is nontrivially unstable, for otherwise \cref{val6triv} applies. Then $X$ must be listed in \cref{val6circulant}. More precisely, since $|G|$ is not divisible by~$8$, it must be listed in \fullref{val6circulant}{C1Se4}, \fullref{val6circulant}{C2}, or~\fullref{val6circulant}{C4}. We will look at each of these possibilities separately (similarly to how these cases were considered in the proof of \cref{val6abcCyclic}).

\smallbreak \textbf{\fullref{val6circulant}{C1Se4}}:
Since the generator $\pm a$ is odd, but the other two are even, the equation $a + b + c = 0$ cannot be satisfied.

\smallbreak \textbf{\fullref{val6circulant}{C2}}:
Since $S$ must contain the element~$k$ of order~$4$, but $a \equiv 0 \pmod 4$, we must have 
	\[ \text{$a + k = k$ \ or \ $a + k = -k$ \ or \ $a - k = k$ \ or \ $a - k = -k$,} \]
so $a \in \{0, 2k \}$. However, we know $a \neq 0$. And $a \neq 2k$, because $2k \not\equiv 0 \pmod{4}$ (or because we are assuming that no element of~$S$ has order~$2$).

\smallbreak \textbf{\fullref{val6circulant}{C4}}:
Since $m$ is odd, we know that $b$ and $mb + 2k$ have the same parity. Therefore, the equation $a + b + c = 0$ implies that $a$ is even. So $a \neq \pm k$. Then $\{\pm b, \pm mb + 2k\}$ contains~$k$, and is therefore contained in~$\langle k \rangle$. This is impossible, because $\langle k \rangle$ does not contain 4 distinct nonzero elements.

\begin{case} \label{val6pf-aut}
Assume $\alpha$ is a group automorphism.
\end{case}
Then 
	\begin{itemize}
	\item $\alpha(g,0) = \bigl( \varphi(g), 0 \bigr)$, for some automorphism~$\varphi$ of~$G$,
	and
	\item $\alpha(0,1) = (\invol, 1)$, for some element~$\invol$ of order~$2$.
	\end{itemize}
Then, since $\alpha \bigl(S \times \{1\} \bigr) = S \times \{1\}$, we have $\varphi(S) = S + \invol$, so $\varphi$ is an isomorphism from $\Cay(G; S)$ to $\Cay(G; S + \invol)$.

We have $\varphi(s) \in \pm\{a,b,c\} + \invol$ for all $s \in \{a,b,c\}$. At least two elements of~$S$ must use the same sign, which implies there exists $\epsilon \in \{0,1\}$, such that $\varphi(S)$ contains at least two elements of $\{\epsilon a + \invol, \epsilon b + \invol, \epsilon c + \invol \}$. We may assume $\epsilon = 1$ (by composing with the automorphism $x \mapsto -x$ if necessary). Then, by permuting $\{a,b,c\}$, we may assume that $a + \invol$ and $b + \invol$ are in $\varphi(\{a,b,c\})$. Therefore
	\[ 0 = \varphi(a + b + c) 
		=  \varphi(a) +  \varphi(b) + \varphi(c) 
		= (a + \invol) + (b + \invol) \pm (c + \invol)
		= (a + b) \pm c + \invol
		= -c \pm c + \invol
		. \]
Since $\invol \neq 0$, we conclude that $2c = \invol$, so $c$ has order~$4$. Therefore, \cref{val6pf-order4} applies.

\begin{case} \label{val6pf-neq}
Assume $2s \neq 2t$, for all $s,t \in S$, such that $s \neq t$.
\end{case}
We may assume $\alpha$ is not a group automorphism, for otherwise \cref{val6pf-aut} applies. 
Therefore, \cref{4cyclenormal} implies there exist $s,t,u,v\in S$ such that  $s + t = u + v \neq 0$ and $\{s,t\} \neq \{u,v\}$. From the assumption of this \lcnamecref{val6pf-neq}, we see that this implies either $3a = \pm c$ or $2a = \pm b \pm c$ (perhaps after permuting $a$, $b$, and~$c$). 
	\begin{itemize}
	\item If $3a = \pm c$, then $c = \pm 3a \in \langle a \rangle$, so $G = \langle a \rangle$. Therefore \cref{val6pf-cyclic} applies.
	\item If $2a = \pm b - c = \pm b + (a + b)$, then $a = \pm b + b \in \langle b \rangle$, so $G = \langle b \rangle$. Therefore \cref{val6pf-cyclic} applies.
	\end{itemize}
So we may assume 
	\[ 2a = \pm b + c = \pm b - (a + b) \in \{-a, -a - 2b\}, \]
so $3a \in \{0, -2b\}$.
	\begin{itemize}
	\item If $3a = -2b$, then $a = -2a - 2b = -2(a + b) = 2c \in \langle c \rangle$, so $G = \langle c \rangle$. Therefore \cref{val6pf-cyclic} applies.
	\end{itemize}
Hence, we may assume 
	\[ \text{$3a = 0$, so $|a| = 3$}. \]
We may also assume $a \notin \langle b \rangle$, for otherwise \cref{val6pf-cyclic} applies. So $\langle a \rangle \cap \langle b \rangle = \{0\}$. This implies $G = \langle a \rangle \times \langle b \rangle$, so, letting $n = |b|$, we have
	\[ X \iso \Cay \bigl( \ZZ_3 \times \ZZ_n; \pm(1,0), \pm(0,1), \pm (1,1) \bigr) . \]

Note that:
	\begin{itemize}
	\item $(\bx a^3)$ and $(\bx a^{-3})$ are paths of length~$3$ in $BX$ from $(0,0)$ to $\bx 0 = (0,1)$,
	\item$(\bx b^3)$ and $(\bx c^{-3})$ are paths of length~$3$ in $BX$ from $(0,0)$ to $3\bx b$,
	and
	\item $(\bx b^{-3})$ and $(\bx c^3)$ are paths of length~$3$ in $BX$ from $(0,0)$ to $3\bx c$. 
	\end{itemize}
Also note that if $s_1$, $s_2$, and~$s_3$ are not all equal to each other, then the number of permutations of the list $(\bx{s_1}, \bx{s_2}, \bx{s_3})$ is divisible by~$3$. Therefore, $\bx 0$, $3 \bx b$, and $3 \bx c$ are the only vertices of~$BX$ for which the number of paths of length~$3$ from $(0,0)$ to the vertex is not divisible by~$3$.  Hence, the set $\{\bx 0, 3 \bx b, 3 \bx c\}$ is $\alpha$-invariant.

We may assume $|b| > 12$, for otherwise \cref{Z3xZn} applies. Then the number of paths of length~$6$ from $3 \bx c = -3 \bx b$ to $3 \bx b$ is $\binom{6}{3} + 2$: 
	\begin{itemize}
	\item $\binom{6}{3}$ permutations of the path $(\bx b^3, \bx c^{-3})$, 
	and
	\item the paths $(\bx b^6)$ and $(\bx c^{-6})$.
	\end{itemize}
This is much smaller than the number of paths of length~$6$ from $\bx 0$ to either $3\bx b$ or~$3\bx c$. For example, paths from $\bx 0$ to $3\bx b$ include:
	\begin{itemize}
	\item $\binom{6}{3}$ permutations of the path $(\bx a^3, \bx b^3)$,
	\item $\binom{6}{3}$ permutations of the path $(\bx a^3, \bx c^{-3})$,
	and
	\item others.
	\end{itemize}
Thus, the vertex $(0,1) = \bx 0$ is uniquely determined as an element of the $\alpha$-invariant set $\{\bx 0, 3 \bx b, 3 \bx c\}$, so it must be fixed by~$\alpha$. This contradicts the choice of~$\alpha$.

\begin{case} \label{val6pf-remain}
The remaining case.
\end{case}
Since \cref{val6pf-neq} does not apply, we have $2s = 2t$, for some $s,t \in S$, such that $s \neq t$. 
Since \cref{val6pf-order4} does not apply, we know $s \neq -t$. Therefore, we may assume $s = a$ and $t \in \{\pm b\}$, so $2a \in \{\pm 2b\}$. 
However, if $2a = -2b$, then 
	\[ 2c = -2(a + b) = -(2a + 2b) = -(-2b + 2b) = -0 = 0, \]
so \cref{val6pf-order2} applies.
Thus, we must have $2a = 2b$. So $b = a + \invol$ for some element~$\invol$ of order~$2$.

We may assume $G \neq \langle a \rangle$ and $G \neq \langle b \rangle$ (otherwise, \cref{val6pf-cyclic} applies), so $b \notin \langle a \rangle$ and $a \notin \langle b \rangle$. Hence $\invol \notin \langle a \rangle$ and $\invol \notin \langle b \rangle$. So
	\[ X \iso \Cay \bigl( \ZZ_n \times \ZZ_2; \pm(1,0), \pm(1,1), \pm(2,1) \bigr) , \]
where $n = |a| = |b|$ is even and, up to a group isomorphism, $a = (1,0)$, $b = (1,1)$, and $c = -(2,1)$.

We may assume $X$ does not satisfy condition~\pref{val6-2a=2b}, so
	\[ \text{$n/2$ is odd} . \]
Also, since \cref{val6pf-order4} (and \cref{val6pf-order2}) does not apply, we have $n > 4$. In addition,
since \cref{val6pf-neq} does not apply, we know $2a \neq -2c$, so $n \neq 6$. Therefore (since $8/2$ is not odd), we have
	\[ n > 8 . \]

Note that, for every $v \in V(BX)$, the fact that $2a = 2b$ implies that $v + 2 \, \bx c$ and $v -2 \, \bx c$ are the only vertices of~$BX$ that are joined to $v$ by a \emph{unique} path of length~$2$. It follows from this that 
	\begin{align} \label{val6pf-cedges}
	 \text{$\alpha(v + k \bx c) = \alpha(v) \pm k \bx c$, \ for all $k \in \ZZ$ and all $v \in V(BX)$.}  
	 \end{align}
Since $\alpha$ fixes~$0$, we conclude that $\{\pm k \, \bx c\}$ is $\alpha$-invariant, for all $k \in \ZZ$.

Let $Y$ be the spanning subgraph of~$BX$ that is obtained by removing all edges of the form $(v, v + \bx c)$. By~\pref{val6pf-cedges}, we know that $Y$ is $\alpha$-invariant, so $\alpha$ is an automorphism of~$Y$.

Now, $\bx{(0,0)}$ and $\bx{(0,1)}$ are the only vertices that are at distance~$2$ from both $\bx c$ and $-\bx c$ in~$Y$. Since $\{ \pm \bx c\}$ is $\alpha$-invariant, this implies that $\bigl\{ \bx{(0,0)}, \bx{(0,1)} \bigr\}$ is also $\alpha$-invariant. 

Since $|\bx c| = n$, we have $(n/2) \bx c = - (n/2) \bx c$. Therefore, we see from~\pref{val6pf-cedges} that $(n/2) \bx c$ is fixed by~$\alpha$. 
However, since $n/2$ is odd, we have
	\[ \bx{(0,1)} = \frac{n}{2} \bx{(2,1)} = \frac{n}{2} \, \bx c . \]
Hence, $\bx{(0,1)}$ is fixed by~$\alpha$. Therefore, since $\bx{(0,0)}$ is the only other element of the invariant set $\bigl\{ \bx{(0,0)}, \bx{(0,1)} \bigr\}$, we conclude that $\bx{(0,0)}$ is also fixed by~$\alpha$. This contradicts the choice of~$\alpha$.
\end{proof}

\end{document}